\documentclass[a4paper,12pt]{amsart}

\usepackage[headings]{fullpage}

\usepackage{amsfonts,graphics,amsmath,amsthm,amscd,amssymb,latexsym,euscript,enumerate}
\usepackage{epsfig}
\usepackage{flafter}
\usepackage[all,cmtip,line]{xy}
\usepackage{array}
\usepackage[english]{babel}
\usepackage{overpic}
\usepackage{subfig}
\usepackage{multirow}
\usepackage{microtype}
\usepackage{wrapfig}
\usepackage{tabularray}
\usepackage{longtable}
\usepackage{supertabular}
\usepackage{caption}
\usepackage{tikz}
\usetikzlibrary{positioning}
\usepackage{tikz-cd}
\usepackage{float}
\allowdisplaybreaks
\usepackage[dvipsnames,svgnames,table]{xcolor}
\usepackage{graphicx}
\usepackage{mathtools}
\usepackage{enumitem}
\usepackage{hyperref}
\hypersetup{
    colorlinks=true,
    linkcolor=blue,
    citecolor=blue,
    filecolor=blue,
    urlcolor=blue
}

\newtheorem{theorem}{Theorem}[section]
\newtheorem{lemma}[theorem]{Lemma}
\newtheorem{proposition}[theorem]{Proposition}

\newtheorem*{theorem*}{Theorem}

\theoremstyle{plain}
\newtheorem{corollary}[theorem]{Corollary}

\theoremstyle{definition}
\newtheorem{definition}[theorem]{Definition}
\newtheorem{definition-lemma}[theorem]{Definition-Lemma}

\numberwithin{equation}{section}

\newcommand{\Q}{\mathbb{Q}}
\newcommand{\OO}{\mathcal{O}}

\def\P{\mathbb{P}}
\newcommand{\A}{\mathbb{A}}
\newcommand{\F}{\mathbb{F}}
\newcommand{\G}{\mathbb{G}}

\DeclareMathOperator{\im}{Image}

\DeclareMathOperator{\Cl}{Cl}

\DeclareMathOperator{\Spec}{Spec}
\DeclareMathOperator{\Supp}{Supp}
\DeclareMathOperator{\SAut}{SAut}
\DeclareMathOperator{\Affcone}{Affcone}

\DeclareMathOperator{\Aut}{Aut}

\DeclarePairedDelimiterX{\norm}[1]{\lVert}{\rVert}{#1}

\title[Flexibility of affine cones over $S_m^n$]
{Flexibility of affine cones over blow-ups of weighted projective planes}

\begin{document}

\author[I.-K. Kim]{In-Kyun Kim}
\author[D.-W. Lee]{Dae-Won Lee}
\author[M. Sawahara]{Masatomo Sawahara}
\address[In-Kyun Kim]{June E Huh Center for Mathematical Challenges, Korea Institute for Advanced Study, 85 Hoegiro, Dongdaemun-gu, Seoul 02455, Republic of Korea}
\email{soulcraw@kias.re.kr}
\address[Dae-Won Lee]{School of Mathematics, Korea Institute for Advanced Study, 85 Hoegiro, Dongdaemun-gu, Seoul 02455, Republic of Korea}
\email{daewonlee@kias.re.kr}
\address[Masatomo Sawahara]{Faculty of Education, Hirosaki University, Bunkyocho 1, Hirosaki-shi, Aomori 036-8560, Japan}
\email{sawahara.masatomo@gmail.com}

\thanks{The first and second authors are supported by the National Research Foundation of Korea (No. RS-2025-00513064). The second author is partially supported by the Basic Science Research Program through the National Research Foundation of Korea (NRF), funded by the Ministry of Education (No. RS-2023-00237440 and 2021R1A6A1A10039823), and by the Samsung Science and Technology Foundation under Project No. SSTF-BA2302-03. The third author is supported by JSPS KAKENHI Grant No. JP24K22823 and JP25K17222.}

\subjclass[2020]{14R20, 14J26, 14J45, 14L30}
\date{\today}
\keywords{affine cones, flexibility, generic flexibility, additive group actions, weighted projective planes, Danielewski surfaces}

\begin{abstract}
Let $S_m^n$ be the surface obtained by blowing up $\P(1,1,m)$ at $n$ general smooth points, where $m\geq2$. For every very ample Cartier divisor $H$ on $S_m^n$, we prove that $\Affcone_H(S_m^n)$ is flexible when $n\leq m+1$. For $n=m+2$ and $n=m+3$, we prove that $\Affcone_H(S_m^n)$ is generically flexible. 
\end{abstract}

\maketitle

Throughout the paper, all varieties are defined over an algebraically closed field $\Bbbk$ of characteristic $0$.

\section{Introduction}\label{sect:intro}
Let $V$ be an affine variety. A point $v\in V_{\rm reg}$ is called \emph{flexible} if $T_vV$ is spanned by tangent vectors to orbits of one-parameter unipotent subgroups of $\Aut(V)$. The variety $V$ is called \emph{flexible} if every point of $V_{\rm reg}$ is flexible. We denote by $\SAut(V)$ the subgroup of $\Aut(V)$ generated by all one-parameter unipotent subgroups.

\begin{theorem}[{\cite[Theorem 0.1]{AFKKZ13}}]\label{thm:AFKKZ}
Let $V$ be an irreducible affine variety of dimension at least $2$. The following assertions are equivalent.
\begin{enumerate}[label=\textup{(\arabic*)}]
\item The variety $V$ is flexible.
\item The group $\SAut(V)$ acts transitively on $V_{\rm reg}$.
\item The group $\SAut(V)$ acts infinitely transitively on $V_{\rm reg}$.
\end{enumerate}
\end{theorem}

For a normal projective variety $Y$ and a very ample Cartier divisor $H$, set
\[
\Affcone_H(Y)
\coloneqq
\Spec\left(\bigoplus_{d\geq0}H^0\bigl(Y,\OO_Y(dH)\bigr)\right).
\]

Micha{\l}ek--Perepechko--S\"u\ss{} proved a criterion for flexibility in terms of coverings by flexible $H$-polar affine open subsets \cite{MPS18}. Perepechko proved an analogous criterion for the existence of an open $\SAut$-orbit \cite{Per21}. We call an affine variety $V$ \emph{generically flexible} if $V_{\rm reg}$ contains a nonempty open subset consisting of flexible points. Equivalently, $V$ is generically flexible if $\SAut(V)$ has an open orbit in $V_{\rm reg}$. We call this property \emph{generic flexibility}.

For smooth del Pezzo surfaces, flexibility and generic flexibility of affine cones are known in several low degrees. Let $S$ be a smooth del Pezzo surface. When $(-K_S)^2=5$, Perepechko proved that the affine cone associated with every very ample divisor is flexible \cite{Per13}. When $(-K_S)^2=4$, flexibility for every ample polarization is proved in \cite{PW16}. When $(-K_S)^2=3$, every affine cone associated with a very ample divisor not proportional to $-K_S$ is generically flexible \cite{Per21}. For $(-K_S)^2=2$, generic flexibility was established in \cite{KP21} for the ample polarizations to which the polar cylinder configurations constructed there apply. These results provide the smooth counterpart to the singular del Pezzo surfaces studied below.

\begin{theorem}\label{thm:main}
Let $m\geq2$, let $S_m^n$ be obtained by blowing up $\P(1,1,m)$ at $n$ general smooth points, and let $H$ be a very ample Cartier divisor on $S_m^n$.
\begin{enumerate}[label=\textup{(\arabic*)}]
\item If $0\leq n\leq m+1$, then $\Affcone_H(S_m^n)$ is flexible.
\item If $n=m+2$, then $\Affcone_H(S_m^{m+2})$ is generically flexible.
\item If $n=m+3$, then $\Affcone_H(S_m^{m+3})$ is generically flexible.
\end{enumerate}
\end{theorem}

The cases $n\leq2$ of Theorem \ref{thm:main}\textup{(1)} also follow from the flexibility of nondegenerate affine toric varieties, since $S_m^n$ is toric for $n=0,1,2$; see \cite[p. 469--470]{CP20} and \cite[Theorem 0.2]{AZK12}. Moreover, for a del Pezzo surface of degree $4$ with an $A_1$-singularity, flexibility of affine cones associated with ample divisors of birational type was proved in \cite{Won22}. The result of \cite{Won22} gives flexibility for that subfamily, whereas Theorem \ref{thm:main}(2) gives generic flexibility for every very ample Cartier divisor.

An open subset of $S_m^n$ is called a {\it cylinder} if it is isomorphic to $\A^1 \times Z$ for some affine curve $Z$. For an ample $\Q$-divisor $H$ on $S_m^n$, a cylinder $U$ in $S_m^n$ is an {\it $H$-polar cylinder} if there exists an effective $\Q$-divisor $D$ on $S_m^n$ such that $D \sim_{\Q} H$ and $U = S_m^n \setminus \Supp (D)$. 
The surfaces $S_m^n$ have already been studied extensively from the viewpoint of polarized cylinders. By \cite[Theorem 4.1]{Saw25}, if $n\leq m+3$, then every ample $\Q$-divisor $H$ on $S_m^n$ admits an $H$-polar cylinder (see also \cite[Remark 4.10]{KKW25}). On $S_m^{m+4}$, there is no anticanonical polar cylinder \cite{KKW25}. More recently, \cite{KSW26} obtained sufficient conditions for an ample $\Q$-divisor $H$ on $S_m^{m+4}$ to admit an $H$-polar cylinder: 
their result depends on whether the corresponding Fujita decomposition is of type $B(r_H)$ or type $C(\ell_H)$. Over an arbitrary field $\Bbbk$ of characteristic zero, \cite{KLS26} classified the cylindricity of $\Bbbk$-forms of these weighted singular del Pezzo surfaces. 

The results above concern the existence of individual polar cylinders. The present paper instead constructs collections of flexible $H$-polar affine open subsets and uses the associated additive group actions to study the special automorphism group of the affine cone.

The proof constructs effective $\Q$-divisors $D\sim_{\Q}H$ whose complements are isomorphic either to $\A^2$ or to smooth Danielewski surfaces. These affine open subsets cover $(S_m^n)_{\rm reg}$ when $n\leq m+1$. If $n=m+2$ or $n=m+3$, then the affine open subsets form a transversal collection, and no $\Q$-divisor $\Q$-linearly equivalent to $H$ is supported on the complement of their union.

The rest of this paper is organized as follows. In Section \ref{sect:prelim}, we recall the criteria for flexibility and for the existence of an open $\SAut$-orbit, together with basic properties of Danielewski surfaces and the divisor classes on $S_m^n$. In Section \ref{sect:main}, after giving a common construction of $H$-polar affine open subsets, we prove all three assertions of Theorem \ref{thm:main}.

\section{Preliminaries}\label{sect:prelim}
\subsection{Notation}\label{subsect:not}
Let $m\geq2$ and $n\geq0$. Let $\beta\colon S_m^n\to\P(1,1,m)$
be the blow-up at $n$ general points $p_1,\dots,p_n$ in the smooth locus; for $n=0$, we set $S_m^0\coloneqq\P(1,1,m)$. 
Note that $S_m^n$ has a unique singular point of type $\frac{1}{m}(1,1)$. 
Denote by $E_i$ the $\beta$-exceptional curve over $p_i$. Let $\pi\colon\widetilde S_m^n\to S_m^n$
be the minimal resolution, and let $\widetilde Q$ be its exceptional curve. 
The exceptional curve $\widetilde Q$ satisfies $\widetilde Q^2=-m$.
Since the points $p_i$ lie in the smooth locus, resolving the unique singular point of $S_m^n$ commutes with these blow-ups. Hence, the strict transforms $\widetilde E_i\coloneqq\pi_{\ast}^{-1}(E_i)$ are pairwise disjoint $(-1)$-curves. Contracting them gives a morphism $\sigma\colon\widetilde S_m^n\to \F_m$. Put $\overline{p}_i\coloneqq\sigma(\widetilde E_i)$ and $\overline Q\coloneqq\sigma_{\ast}(\widetilde Q)$. Let $\tau\colon\F_m\longrightarrow\P^1$
be the ruling, and let $\overline F$ denote its fiber class.
If $\vartheta\colon\F_m\to\P(1,1,m)$ denotes the contraction of the negative section, then we have the following diagram.
\[
\begin{tikzcd}
\widetilde S_m^n \arrow[r,"\sigma"] \arrow[d,"\pi"']
& \F_m \arrow[d,"\vartheta"] \arrow[r,"\tau"] & \P^1 \\
S_m^n \arrow[r,"\beta"']
& \P(1,1,m)
\end{tikzcd}
\]
 
We assume that the blown-up points are in general position; in particular, the points $\overline{p}_i\in\F_m\setminus\overline Q$ lie on pairwise distinct fibers. Set $\widetilde F\coloneqq\sigma^{\ast}(\overline F)$. For each $i$, let $\overline F_i$ be the fiber through $\overline{p}_i$, put $\widetilde E_i'\coloneqq\sigma_{\ast}^{-1}(\overline F_i)$, and set $F\coloneqq\pi_{\ast}(\widetilde F)$ and $E_i'\coloneqq\pi_{\ast}(\widetilde E_i')$.
Then $\widetilde E_i'\sim\widetilde F-\widetilde E_i$.

We use the same symbol for an irreducible curve and its divisor class whenever no confusion is likely to arise. One has
\[
\Cl(S_m^n)_{\Q}=\Q[F]\oplus\bigoplus_{i=1}^n\Q[E_i].
\]
For a very ample Cartier divisor $H$ on $S_m^n$, write
\begin{equation}\label{eq:Hclass}
H\sim_{\Q}aF-\sum_{i=1}^nc_iE_i
\end{equation}
for some $a,c_1,\dots,c_n \in \Q$. 
Since $\pi^{\ast}F=\widetilde F+\frac1m\widetilde Q$,
\begin{equation}\label{eq:Hpullback}
\pi^{\ast}H
\sim_{\Q}
\alpha\widetilde Q+a\widetilde F-\sum_{i=1}^nc_i\widetilde E_i,
\qquad
\alpha\coloneqq\frac am.
\end{equation}
Ampleness gives
\begin{equation}\label{eq:basic-positive}
c_i=H\cdot E_i>0,
\qquad
\alpha-c_i=H\cdot E_i'>0.
\end{equation}

\subsection{Polar affine open subsets and affine cones}\label{subsect:lifting}
Let $Y$ be a normal projective variety and let $H$ be a very ample Cartier divisor. An affine open subset $U\subseteq Y$ is \emph{$H$-polar} if $U=Y\setminus\Supp(D)$ for an effective $\Q$-divisor $D\sim_{\Q}H$.

Let $A_H$ denote the homogeneous coordinate ring of the image of $Y$ under the embedding defined by $|H|$, and set
\[
R(Y,H)\coloneqq\bigoplus_{d\geq0}H^0\bigl(Y,\OO_Y(dH)\bigr).
\]
Thus, $A_H$ is naturally a graded subring of $R(Y,H)$, with 
\[
A_{H,d}=\im\left(\mathrm{Sym}^d H^0(Y,\OO_Y(H))\longrightarrow H^0(Y,\OO_Y(dH)\right).
\]
Let $X_H\coloneqq \Spec A_H$. The natural morphism $\Affcone_H(Y)\to X_H$ is the normalization morphism.

\begin{lemma}\label{lem:normalization}
The ring $R(Y,H)$ is the integral closure of $A_H$. The normalization
\[
\nu_H\colon\Affcone_H(Y)\longrightarrow X_H
\]
restricts to an isomorphism between the complements of the vertices. Every algebraic $\G_a$-action on $X_H$ lifts uniquely to $\Affcone_H(Y)$.
\end{lemma}

\begin{proof}
By \cite[Chapter II, Exercise 5.14(a)]{Har77}, the section ring $R(Y,H)$ is the integral closure of $A_H$. Let $0\neq s\in A_{H,1}$, and let $t\in R(Y,H)_d$ be homogeneous. The Serre vanishing implies that the ideal sheaf of the embedding defined by $|H|$ gives $A_{H,e}=R(Y,H)_e$ for all sufficiently large $e$. Choose $k\gg 0$ such that $A_{H,d+k}=R(Y,H)_{d+k}$. Then $s^kt\in A_{H,d+k}$ and consequently $t\in (A_H)_s$. Hence, we obtain that $(A_H)_s=R(Y,H)_s$.

Let $\mathfrak m_A\coloneqq \oplus_{d>0} A_{H,d}$ and $\mathfrak m_R\coloneqq \oplus_{d>0} R(Y,H)_d$ be the homogeneous maximal ideals defining the vertices. If $t\in R(Y,H)$ is homogeneous of positive degree, then $t^N\in A_H$ for $N\gg 0$, and $t^N\in \mathfrak m_A R(Y,H)$. Therefore, we have $\sqrt{\mathfrak m_A R(Y,H)}=\mathfrak m_R$. Since $A_H$ is generated by $A_{H,1}$, the principal open subsets $D(s)$ with $0\neq s\in A_{H,1}$ cover $\Affcone_H(Y)\setminus \{\mathfrak m_R\}$. Since we have $(A_H)_s=R(Y,H)_s$, the normalization $\nu_H$ restricts to an isomorphism between the complements of the vertices.

The lifting statement follows from the functoriality of normalization with respect to algebraic group actions.
\end{proof}

The following proposition allows us to deduce flexibility of the affine cone from a covering of the smooth locus by flexible $H$-polar affine open subsets.

\begin{proposition}[{cf. \cite[Theorem 1.4 and Corollary 1.2]{MPS18}}]\label{prop:flex-cover}
Let $Y$ be a normal projective variety, and let $H$ be a very ample Cartier divisor on $Y$. 
Assume that there exist finitely many smooth flexible $H$-polar affine open subsets $U_1,\dots,U_r$ of $Y$ such that $Y_{\rm reg}=\bigcup_{i=1}^rU_i$. 
Then $\Affcone_H(Y)$ is flexible.
\end{proposition}

\begin{proof}
If $X_H$ is an affine space, then $A_H$ is normal, and Lemma \ref{lem:normalization} gives $\Affcone_H(Y)=X_H$. Hence, the assertion is immediate. 

Assume that $X_H$ is not an affine space. By the construction in the proof of \cite[Theorem 1.4]{MPS18}, finitely many $\G_a$-actions on the $U_i$ lift to homogeneous $\G_a$-actions on $X_H$. Let $G\subseteq \SAut(X_H)$ be the subgroup generated by the lifted actions. Since any two nonempty open subsets of the irreducible variety $Y$ intersect and $\SAut(U_i)$ acts transitively on $U_i$, the orbit correspondence gives a point $x\in(X_H)_{\rm reg}$ such that $\pi_H(Gx)=Y_{\rm reg}$,
where $\pi_H\colon X_H\setminus\{0\}\to Y$ is the natural projection. Since $Y$ is normal, by \cite[Corollary 1.2]{MPS18}, we have $Gx=(X_H)_{\rm reg}$.

By Lemma \ref{lem:normalization}, every one-parameter unipotent subgroup used above lifts uniquely to $\Affcone_H(Y)$, and the normalization is an isomorphism away from the vertices. If the vertex of $\Affcone_H(Y)$ is singular, the lifted subgroup acts transitively on $(\Affcone_H(Y))_{\rm reg}$.

Suppose that the vertex of $\Affcone_H(Y)$ is smooth, and let $\mathfrak m$ be its homogeneous maximal ideal. Choose homogeneous elements $f_1,\dots,f_N\in \mathfrak{m}$ whose classes form a basis of $\mathfrak m/\mathfrak m^2$. Since the classes of the $f_i$ generate $\mathfrak{m}/\mathfrak{m}^2$, induction on the degree shows that $f_1,\dots,f_N$ generate $R(Y,H)$ as a $\Bbbk$-algebra. Since the vertex is smooth, we have $N=\dim_{\Bbbk}\mathfrak m/\mathfrak m^2=\dim\Affcone_H(Y)$. The surjective homomorphism $\Bbbk[z_1,\dots,z_N]\to R(Y,H)$ has a kernel which is a prime ideal of height zero and is therefore zero. Hence, $\Affcone_H(Y)$ is an affine space. In either case, $\Affcone_H(Y)$ is flexible.
\end{proof}

\begin{definition}[{cf. \cite[Definition 2.3]{Per21}}]\label{def:transversal}
Let $\mathcal U$ be a collection of affine open subsets of $Y$.
\begin{enumerate}[label=\textup{(\arabic*)}]
\item A subset $Z\subseteq \bigcup_{U\in \mathcal U} U$ is {\it $\mathcal U$-invariant} if $Z\cap U$ is $\SAut(U)$-invariant for every $U\in\mathcal U$.
\item The collection $\mathcal U$ is {\it transversal} if $\bigcup_{U\in\mathcal U}U$ has no nonempty proper $\mathcal U$-invariant subset.
\end{enumerate}
\end{definition}

\begin{theorem}[{\cite[Theorem 2.4]{Per21}}]\label{thm:Per}
Let $Y$ be a normal projective variety, let $H$ be a very ample divisor on $Y$, and let $\mathcal U$ be a transversal collection of $H$-polar affine open subsets of $Y$. Let $\pi_H\colon X_H\setminus \{0\}\to Y$ be the projection. Then there exists an $\SAut$-orbit $\mathcal{O}\subseteq X_H$ such that $\pi_H(\mathcal{O})\supseteq \bigcup_{U\in \mathcal{U}} U$. If no $\Q$-divisor $D\sim_{\Q}H$ satisfies
\[
\Supp(D)\subseteq Y\setminus\bigcup_{U\in\mathcal U}U,
\]
then $\mathcal{O}$ is open and contains $\pi_H^{-1}(\bigcup_{U\in \mathcal U} U)$.
\end{theorem}


Transversality also follows from a covering by smooth flexible affine open subsets.

\begin{lemma}\label{lem:connected-transversal}
Let $Y$ be an irreducible projective variety, and let $\mathcal U$ be a collection of smooth flexible affine open subsets of $Y$. Then $\mathcal U$ is transversal.
\end{lemma}

\begin{proof}
Let $Z$ be a nonempty $\mathcal U$-invariant subset of $\bigcup_{U\in\mathcal U}U$, and choose $U_0\in\mathcal U$ such that $Z\cap U_0\neq\emptyset$. Since $U_0$ is flexible, $\SAut(U_0)$ acts transitively on $U_0$, and hence $U_0\subseteq Z$. For every $U\in\mathcal U$, the irreducibility of $Y$ implies that $U_0\cap U\neq\emptyset$. Thus, $Z\cap U\neq\emptyset$, and the transitivity of $\SAut(U)$ implies $U\subseteq Z$. Therefore, $Z=\bigcup_{U\in\mathcal U}U$, and $\mathcal U$ is transversal.
\end{proof}

\begin{corollary}\label{cor:open-orbit}
    Let $Y$ be a normal projective variety, let $H$ be a very ample Cartier divisor on $Y$, and let $\mathcal U$ be a collection of smooth flexible $H$-polar affine open subsets of $Y$. Assume that no $\Q$-divisor $D\sim_{\Q} H$ satisfies
    \[
    \Supp(D)\subseteq Y\setminus \bigcup_{U\in \mathcal U}U.
    \]
    Then $\Affcone_H(Y)$ has an open $\SAut$-orbit contained in $\Affcone_H(Y)_{\rm reg}$. In particular, $\Affcone_H(Y)$ is generically flexible.
\end{corollary}
\begin{proof}
Lemma \ref{lem:connected-transversal} shows that $\mathcal U$ is transversal. If $X_H$ is an affine space, then by Lemma \ref{lem:normalization} $\Affcone_H(Y)=X_H$ and the conclusion is immediate.
    
By Theorem \ref{thm:Per}, the affine cone $X_H$ has an open $\SAut(X_H)$-orbit $\mathcal O$ containing $\pi_H^{-1}\left(\bigcup_{U\in\mathcal U}U\right)$.
For every $U\in\mathcal U$, the variety $\pi_H^{-1}(U)$ is a $\G_m$-bundle over the smooth variety $U$. Hence, $\mathcal{O}$ meets $(X_H)_{\rm reg}$. Since automorphisms preserve the regular locus, $\mathcal O\subseteq(X_H)_{\rm reg}$.

By Lemma \ref{lem:normalization}, the $\G_a$-actions used to construct $\OO$ lift uniquely to $\Affcone_H(Y)$. If the vertex of $\Affcone_H(Y)$ is singular, then $\nu_H^{-1}(\mathcal O)$ is the required open orbit. If the vertex is smooth, then the proof of Proposition \ref{prop:flex-cover} shows that $\Affcone_H(Y)$ is an affine space.
\end{proof}

\subsection{Danielewski surfaces}\label{subsect:danielewski-prelim}
For a nonconstant polynomial $P\in\Bbbk[x]$, the affine surface
\[
U_P\coloneqq\Spec\frac{\Bbbk[u,v,x]}{(uv-P(x))}
\]
is called a \emph{Danielewski surface}. Surfaces of this type were first introduced in an unpublished preprint of Danielewski \cite{Dan89}, where they arose in connection with the cancellation problem. See also \cite{Fie94} for a systematic treatment of complex affine surfaces with additive group actions. The coordinate functions $u$ and $v$ each define an $\A^1$-fibration $U_P\to\A^1$. If $P$ is square-free, then both $u^{-1}(0)$ and $v^{-1}(0)$ are disjoint unions of affine lines indexed by the roots of $P$. For the geometry and algebraic vector fields on these surfaces, see \cite{Leu16}.

If $P$ is square-free, then the Jacobian criterion shows that $U_P$ is smooth, and the suspension theorem \cite[Theorem 0.2]{AZK12} implies that it is flexible. If $\deg P=1$, then by eliminating $x$, we have $U_P\simeq\A^2$.

\section{Main results and proofs}\label{sect:main}
In this section, we keep the notation from Subsection \ref{subsect:not}. 
\subsection{A common construction of $H$-polar affine open subsets}\label{subsect:danielewski}
Let $\widetilde B\subset \widetilde S_m^n$ be a smooth section satisfying
\[
\widetilde B
\sim
\widetilde Q+d\widetilde F-\sum_{i=1}^n\widetilde E_i,
\qquad
d=
\begin{cases}
    m,\quad &(n\leq m+1)\\
    m+1,\quad &(n\in \{m+2, m+3\})
\end{cases}\ldotp
\]
Let $B\coloneqq \pi_{\ast}(\widetilde B)$. 
If $d=m+1$, then assume that the fiber through $\sigma(\widetilde B\cap \widetilde Q)$ contains none of the points $\overline{p}_1,\dots,\overline{p}_n$, and let $\overline{F}_{\infty}$ be this fiber. If $d=m$, then choose a fiber of $\tau$ $\overline{F}_{\infty}$ disjoint from $\overline{p}_1,\dots,\overline{p}_n$. Let $\infty\coloneqq \tau(\overline{F}_{\infty})$ and choose an affine coordinate $x$ on $\P^1\setminus \{\infty\}$. Since the restriction of the ruling over $\P^1\setminus \{\infty\}$ is a trivial $\A^1$-bundle after removing the negative section, choose a trivialization
\[
\psi\colon \F_m\setminus (\overline{Q}\cup\overline{F}_{\infty})\xrightarrow{\simeq} \A_{x,b}^2\coloneqq \Spec \Bbbk[x,b]
\]
for which the first coordinate is the base coordinate $x$. The second coordinate $b$ is therefore a coordinate along the fibers of the ruling. By translating the fiber coordinate, we may assume that
\[
\psi\left(\sigma_{\ast}(\widetilde B)\cap (\F_m\setminus (\overline{Q}\cup\overline{F}_{\infty}))\right)=V(b).
\]
For every $i$, define $\lambda_i\coloneqq x(\tau(\overline{p}_i))\in \Bbbk$. Then $\psi(\overline{p}_i)=(\lambda_i,0)$. The numbers $\lambda_1,\dots,\lambda_n$ are pairwise distinct since the points $\overline{p}_i$ lie on pairwise distinct fibers of $\tau$. Put $\widetilde F_\infty\coloneqq\sigma_{\ast}^{-1}(\overline F_\infty)$ and $F_{\infty}\coloneqq \pi_{\ast}(\widetilde F_{\infty})$.

For $I\subseteq\{1,\dots,n\}$, put $P_I(x)\coloneqq\prod_{i\in I}(x-\lambda_i)$ and $P_\emptyset(x)\coloneqq1$. Thus, $P_I$ is the square-free polynomial whose roots are the base coordinates of the fibers containing the points $\overline{p}_i$ with $i\in I$. For $q\in\{c_1,\dots,c_n\}$, set
\[
L_q\coloneqq\{i\mid c_i<q\},
\qquad
T_q\coloneqq\{i\mid c_i=q\},
\qquad
R_q\coloneqq\{i\mid c_i>q\},
\]
These three sets form a disjoint partition of $\{1,\dots,n\}$. Define
\begin{equation}\label{eq:gamma-general}
\gamma_q
\coloneqq
a-\sum_{j\in R_q}c_j-q\bigl(d-|R_q|\bigr),
\end{equation}
\begin{equation}\label{eq:D-general}
D_q
\coloneqq
qB
+\sum_{i\in L_q}(q-c_i)E_i
+\sum_{j\in R_q}(c_j-q)E_j'
+\gamma_qF_\infty,
\end{equation}
and $U_q\coloneqq S_m^n\setminus \Supp(D_q)$. The roots of $P_{T_q}$ are precisely the base coordinates of the fibers corresponding to the indices $i$ for which neither $E_i$ nor $E_i'$ is a component of $\Supp(D_q)$.

The next proposition gives the complement of the boundary divisor $D_q$ explicitly; $T_q$ determines the Danielewski surface.

\begin{proposition}\label{prop:danielewski}
If $\gamma_q>0$, then $D_q$ is an effective $\Q$-divisor satisfying $D_q\sim_{\Q} H$, and 
\begin{equation}\label{eq:Danielewski}
S_m^n\setminus\Supp(D_q)
\simeq
\Spec\frac{\Bbbk[x,u,v]}{\bigl(uv-P_{T_q}(x)\bigr)}.
\end{equation}
In particular, $U_q=S_m^n\setminus \Supp(D_q)$ is smooth and flexible.
\end{proposition}

\begin{proof}
Since $B\sim_{\Q}dF-\sum_iE_i$ and $E_i'\sim_{\Q}F-E_i$, the coefficient of $E_i$ in $D_q$ is $-c_i$ for every $i$, while the coefficient of $F$ in $D_q$ is
\[
qd+\sum_{j\in R_q}(c_j-q)+\gamma_q=a.
\]
Since $q>0$, $\gamma_q>0$, $q-c_i>0$ for $i\in L_q$, and $c_j-q>0$ for $j\in R_q$, the divisor $D_q$ is effective, and $D_q\sim_{\Q}H$.

Put $P(x)\coloneqq P_{\{1,\dots,n\}}(x)$. The affine modification of $\A^2_{x,b}$ along the divisor $(b=0)$ with center defined by $(b,P(x))$ is
\[
\Spec\Bbbk\left[x,b,\frac{P(x)}b\right]
\simeq
\Spec\frac{\Bbbk[x,b,w]}{(bw-P(x))}.
\]
After removing the curves $\widetilde E_i$ for $i\in L_q$ and $\widetilde E_j'$ for $j\in R_q$, the functions
\[
u\coloneqq\frac{b}{P_{R_q}(x)},
\qquad
v\coloneqq\frac{w}{P_{L_q}(x)}
\]
are regular and satisfy $uv=P_{T_q}(x)$. Conversely, $b=uP_{R_q}(x)$ and $w=vP_{L_q}(x)$, and these give the isomorphism on the inverse image of $U_q$ in $\widetilde S_m^n$. The curve $F_{\infty}$ contains the singular point $\pi(\widetilde Q)$, since $\overline{F}_{\infty}$ meets the negative section $\overline{Q}$. Consequently, the inverse image of $U_q$ under $\pi$ shows that $\pi$ restricts to an isomorphism onto $U_q$. This proves \eqref{eq:Danielewski}. Since $T_q\neq\emptyset$ and the $\lambda_i$ are pairwise distinct, $P_{T_q}$ is nonconstant and square-free. The final assertion follows from Subsection \ref{subsect:danielewski-prelim}.
\end{proof}

\subsection{Proof of Theorem \ref{thm:main}}\label{subsect:proof-main}
We prove Theorem \ref{thm:main} by applying Proposition \ref{prop:danielewski} throughout this subsection.
\medskip

\noindent\textbf{The case $0\leq n\leq m+1$.}
Assume first that $1\leq n\leq m+1$. Choose a fiber $\overline F_\infty$ disjoint from the points $\overline{p}_i$, and put $\widetilde F_{\infty}\coloneqq \sigma_{\ast}^{-1}(\overline F_{\infty})$. Write $\overline{p}_i=(\lambda_i,\eta_i)$ on $\F_m\setminus(\overline Q\cup\overline F_\infty)\simeq\A^2_{x,y}$.
The evaluation map
\[
H^0\bigl(\P^1,\OO_{\P^1}(m)\bigr)
\longrightarrow
\bigoplus_{i=1}^n\Bbbk,
\qquad
h\longmapsto\bigl(h(\lambda_1),\dots,h(\lambda_n)\bigr),
\]
is surjective. Choose $h$ such that $h(\lambda_i)=\eta_i$ for every $i$, and let $\widetilde B$ be the strict transform of the closure of the graph $y=h(x)$. Then
\[
\widetilde B
\sim
\widetilde Q+m\widetilde F-\sum_{i=1}^n\widetilde E_i.
\]
Put $B\coloneqq\pi_{\ast}(\widetilde B)$ and
\begin{equation}\label{eq:delta-small}
\delta
\coloneqq
\pi^{\ast}H\cdot\widetilde B
=a-\sum_{i=1}^nc_i
>0.
\end{equation}

\begin{lemma}\label{lem:A2-small}
Let $\overline F_0$ be a fiber of $\tau$, disjoint from the points $\overline{p}_i$, and let $\widetilde F_0\coloneqq \sigma_{\ast}^{-1}(\overline F_0)$, and put  $F_0\coloneqq\pi_{\ast}(\widetilde F_0)$. Then
\[
D_0(F_0)
\coloneqq
\delta F_0+\sum_{i=1}^nc_iE_i'
\sim_{\Q}H,
\qquad
U_0(F_0)
\coloneqq
S_m^n\setminus\Supp(D_0(F_0))
\simeq\A^2.
\]
\end{lemma}

\begin{proof}
The $\Q$-linear equivalence follows from $E_i'\sim_{\Q}F-E_i$ and \eqref{eq:delta-small}.

Choose an isomorphism $\F_m\setminus(\overline Q\cup\overline F_0)\simeq \A^2_{x,y}$, and write $\overline p_i=(\lambda_i,\eta_i)$ in these coordinates.
Choose $g\in\Bbbk[x]$ such that
$g(\lambda_i)=\eta_i$ for every $i$, and put
\[
P(x)\coloneqq\prod_{i=1}^n(x-\lambda_i).
\]
The morphism
\[
\A^2_{x,z}\longrightarrow\A^2_{x,y},
\qquad
(x,z)\longmapsto\bigl(x,g(x)+P(x)z\bigr),
\]
realizes the affine modification obtained by blowing up the points $(\lambda_i,\eta_i)$ and deleting the strict transforms of the fibers $(x=\lambda_i)$. Therefore,
\[
\widetilde S_m^n\setminus
\left(
\widetilde Q\cup\widetilde F_0
\cup\bigcup_{i=1}^n\widetilde E_i'
\right)
\simeq
\A^2.
\]
Since $F_0$ contains the singular point $\pi(\widetilde Q)$,
the morphism $\pi$ restricts to an isomorphism from the open
subset above onto $U_0(F_0)$. Hence, $U_0(F_0)\simeq\A^2$.
\end{proof}

Let $M\coloneqq\max_i c_i$. For $q\in\{c_1,\dots,c_n\}$ and $r\coloneqq|R_q|$, one has  $r\leq n-1\leq m$ and
\[
\gamma_q
=a-\sum_{j\in R_q}c_j-q(m-r)
\geq a-mM
>0
\]
by \eqref{eq:basic-positive}; indeed, $M<\alpha$ and $a=m\alpha$. Proposition \ref{prop:danielewski} therefore gives an $H$-polar flexible affine open subset $U_q$ for every distinct coefficient value $q$.

Choose two distinct fibers $\overline F_0,\overline F_1$ disjoint from the points $\overline{p}_i$, and put $\widetilde F_j\coloneqq \sigma_{\ast}^{-1}(\overline F_j)$, $F_j\coloneqq\pi_{\ast}(\widetilde F_j)$ for $j=0,1$. The two fibers meet on $S_m^n$ only at the singular point, which lies on every $E_i'$. Moreover, for $q=c_i$, neither $E_i$ nor $E_i'$ occurs in the boundary divisor $D_q$. The section $\sigma_{\ast}(\widetilde B)$ meets $\overline{F}_i$ transversely at $\overline{p}_i$, and therefore its strict transform does not meet $\widetilde E_i'$. The other boundary components lie over fibers distinct from $\overline{F}_i$. Hence,
\[
(S_m^n)_{\rm reg}\setminus\bigcup_{i=1}^nE_i'
\subseteq
U_0(F_0)\cup U_0(F_1),
\qquad
E_i'\cap(S_m^n)_{\rm reg}\subseteq U_{c_i}.
\]
Consequently,
\begin{equation}\label{eq:cover-small}
(S_m^n)_{\rm reg}
=U_0(F_0)\cup U_0(F_1)\cup\bigcup_qU_q.
\end{equation}

If $n=0$, then $H\sim_{\Q}\lambda F$ for some $\lambda>0$. The complements of the images of two distinct fibers are $H$-polar affine planes and cover $\P(1,1,m)_{\rm reg}$.
\medskip

\noindent\textbf{The case $n=m+2$.}
Recall from Subsection \ref{subsect:danielewski} that
\[
\widetilde B\sim \widetilde Q+(m+1)\widetilde F-\sum_{i=1}^{m+2}\widetilde E_i.
\]

\begin{lemma}\label{lem:pencil-middle}
The linear system $|\widetilde B|$ is a base-point-free pencil.
\end{lemma}

\begin{proof}
Put $A\coloneqq\overline Q+(m+1)\overline F$. One has $h^0(\F_m,\OO_{\F_m}(A))=m+4$ and $A^2=m+2$. The $m+2$ general points impose independent conditions on $|A|$, thus the subsystem through them is a pencil. Generality also implies that two generators are smooth at the points $\overline{p}_i$ and meet there transversely. Their total intersection number is $m+2$; hence, the $\overline{p}_i$ are the complete base locus of the pencil, all with multiplicity one. After blowing up the $\overline{p}_i$, the strict transforms of the two generators are disjoint. Therefore, $|\widetilde B|$ is base-point-free.
\end{proof}

For $q\in\{c_1,\dots,c_{m+2}\}$, recall that $\gamma_q$ is defined as in \eqref{eq:gamma-general}.

\begin{lemma}\label{lem:positive-values}
Let $q_1<\cdots<q_s$ be the distinct values among $c_1,\dots,c_{m+2}$, and let $t_k\coloneqq|T_{q_k}|$. Then
\[
\gamma_{q_1}>0,
\qquad
\gamma_{q_{k+1}}-\gamma_{q_k}
=-\left(\sum_{\ell\leq k}t_\ell-1\right)(q_{k+1}-q_k)
\leq0.
\]
\end{lemma}

\begin{proof}
Choose $i\in T_{q_1}$. 
Since $h^0(\F_m,\OO_{\F_m}(\overline Q+m\overline F))=m+2$, there exists a section through the points $\overline{p}_j$ with $j\neq i$. By the generality assumption, this section is smooth and does not pass through $\overline{p}_i$. Its strict transform therefore has class
\[
\widetilde\Gamma_i
\sim
\widetilde Q+m\widetilde F-\sum_{j\neq i}\widetilde E_j.
\]
 Moreover,
\[
\gamma_{q_1}
=\pi^{\ast}H\cdot\widetilde\Gamma_i
=a-\sum_{j\neq i}c_j
>0.
\]
The formula for the successive differences follows directly from \eqref{eq:gamma-general}.
\end{proof}

Let $\mathcal C_H\coloneqq\{q\in\{c_1,\dots,c_{m+2}\}\mid\gamma_q>0\}$. Since $\widetilde B\cdot\widetilde Q=1$ and $\widetilde B\cdot\widetilde E_i=1$ for every $i$, both $\widetilde Q$ and the curves $\widetilde E_i$ are sections of the morphism defined by $|\widetilde B|$. Let $\mathcal B\subseteq|\widetilde B|$ be the open subset of smooth irreducible members $\widetilde B'$ satisfying the following conditions. For such a member, let $\overline{F}_{\widetilde B'}$ be the fiber of $\tau$ through $\sigma(\widetilde B'\cap \widetilde Q)$, and let $\widetilde F_{\widetilde B'}\coloneqq \sigma_{\ast}^{-1}(\overline{F}_{\widetilde B'})$.
\begin{enumerate}[label=\textup{(\arabic*)}]
\item $\sigma_{\ast}(\widetilde B')$ meets $\overline{F}_i$ transversely at $\overline{p}_i$ for every $i$;
\item the fiber $\overline F_{\widetilde B'}$ contains none of the points $\overline{p}_i$.
\end{enumerate}

Let $f_{\widetilde B}\colon\widetilde S_m^{m+2}\to\P^1$ be the morphism defined by $|\widetilde B|$. Since $\widetilde Q$ is a section of $f_{\widetilde B}$, the finite morphism in the Stein factorization of $f_{\widetilde B}$ has a section. The Stein factor is integral, and the image of the section is a closed irreducible subset of the same dimension; hence, the finite morphism is an isomorphism. Thus, the fibers of $f_{\widetilde B}$ are connected. By Bertini's theorem, a general member of $|\widetilde B|$ is smooth.
Since the fibers of $f_{\widetilde B}$ are connected, every smooth member is irreducible. For each $i$, tangency to $\overline{F}_i$ at $\overline{p}_i$ defines at most one member of the pencil, since it is equivalent to the strict transform passing through $\widetilde E_i\cap \widetilde E_i'$. Likewise, at most one member satisfies $\overline{F}_{\widetilde B'}=\overline{F}_i$, since $\widetilde Q$ is a section of $f_{\widetilde B}$. Thus, removing these finitely many members gives a nonempty open subset $\mathcal{B}\subseteq |\widetilde B|$.

For $q\in\mathcal C_H$ and $\widetilde B\in\mathcal B$, let $D_{q,\widetilde B}$ and $U_{q,\widetilde B}$ be the divisor and the affine open subset obtained from Proposition \ref{prop:danielewski}, and set
\[
\mathcal{U}_H\coloneqq \{U_{q,\widetilde B'}\mid q\in \mathcal C_H, \widetilde B'\in \mathcal B\},\quad V_H
\coloneqq
\bigcup_{U\in\mathcal U_H}U.
\]

\begin{lemma}\label{lem:union-middle}
One has
\[
(S_m^{m+2})_{\rm reg}\setminus V_H
=
\bigcup_{\gamma_{c_i}\leq0}\bigl(E_i'\cap(S_m^{m+2})_{\rm reg}\bigr).
\]
\end{lemma}

\begin{proof}
If $\gamma_{c_i}\leq0$ and $q\in\mathcal C_H$, Lemma \ref{lem:positive-values} gives $q<c_i$. Hence, $E_i'$ is a component of every $D_{q,\widetilde B}$. If $\gamma_{c_i}>0$, choose $\widetilde B\in\mathcal B$. For $q=c_i$, neither $E_i$ nor $E_i'$ is a component of $D_{q,\widetilde B}$, and the defining conditions of $\mathcal B$ imply that neither $\widetilde B$ nor $\widetilde F_{\widetilde B}$ meets the smooth part of $\widetilde E_i'$. The remaining boundary components are exceptional curves over points on other fibers or strict transforms of fibers different from the fiber through $\overline{p}_i$; hence, they are also disjoint from the smooth part of $\widetilde E_i'$. Thus, we obtain $E_i'\cap(S_m^{m+2})_{\rm reg}\subseteq U_{c_i,\widetilde B}$.

Let $q_1\coloneqq\min_i c_i$. Lemma \ref{lem:positive-values} gives $q_1\in\mathcal C_H$, and $L_{q_1}=\emptyset$. Let $x\in(S_m^{m+2})_{\rm reg}$ lie outside every $E_i'$, and let $\widetilde x$ be its inverse image on $\widetilde S_m^{m+2}$. Since $|\widetilde B|$ is a base-point-free pencil, exactly one member contains $\widetilde x$.

Moreover, the maps
\[
|\widetilde B|\longrightarrow\widetilde Q,
\qquad
\widetilde B'\longmapsto\widetilde B'\cap\widetilde Q,
\]
and
\[
\widetilde Q\longrightarrow\P^1,
\qquad
z\longmapsto\text{the fiber of the original ruling through }z,
\]
are isomorphisms. Consequently, at most one member $\widetilde B$ satisfies $\widetilde x\in\widetilde F_{\widetilde B}$.
 Choose $\widetilde B\in\mathcal B$ outside these at most two members. Then $x\in U_{q_1,\widetilde B}$.
\end{proof}

The collection $\mathcal{U}_H$ is nonempty since $q_1\in \mathcal{C}_H$ and $\mathcal{B}\neq \emptyset$. Lemma \ref{lem:connected-transversal} shows that $\mathcal{U}_H$ is transversal.

\begin{lemma}\label{lem:no-divisor-middle}
There is no $\Q$-divisor $D\sim_{\Q}H$ with $\Supp(D)\subseteq S_m^{m+2}\setminus V_H$.
\end{lemma}

\begin{proof}
Every prime divisor contained in $S_m^{m+2}\setminus V_H$ is one of the curves $E_j'$ with $\gamma_{c_j}\leq0$. Hence, $D=\sum_{j\in I}d_jE_j'$ for some subset $I$ of these indices. Choose $i_0$ with $\gamma_{c_{i_0}}>0$; such an index exists by Lemma \ref{lem:positive-values}. Since $E_j'\sim_{\Q}F-E_j$, the coefficient of $E_{i_0}$ in $D$ is zero, whereas the coefficient of $E_{i_0}$ in $H$ is $-c_{i_0}\neq0$, a contradiction.
\end{proof}
\medskip

\noindent\textbf{The case $n=m+3$.}
Recall from Subsection \ref{subsect:danielewski} that
\[
\widetilde B \sim \widetilde Q+(m+1)\widetilde F-\sum_{i=1}^{m+3}\widetilde E_i.
\]
In this case, the corresponding linear system has a unique member, which we also denote by $\widetilde B$.

\begin{lemma}[\protect{\cite[Lemma 2.14]{CP20}}]\label{lem:plane-presentation}
For general points $\overline{p}_1,\dots,\overline{p}_{m+3}$, the curve $\widetilde B$ is a $(-1)$-curve. Contracting $\widetilde B,\widetilde E_1',\dots,\widetilde E_{m+3}'$ gives a morphism $\rho\colon\widetilde S_m^{m+3}\longrightarrow\P^2$. The image of $\widetilde Q$ is a smooth conic.
\end{lemma}

Let $L$ denote the pullback of the class of a line in $\P^2$, and let $e_0,e_1,\dots,e_{m+3}$ denote the $\rho$-exceptional curves. We index them so that $e_0=\widetilde B$ and $e_i=\widetilde E_i'$ for $1\leq i\leq m+3$. Put $x_i\coloneqq\rho(e_i)$. Then the points $x_0,\dots,x_{m+3}$ lie on the conic $\rho(\widetilde Q)$, and
\begin{equation}\label{eq:Q-plane}
\widetilde Q
\sim
2L-\sum_{i=0}^{m+3}e_i.
\end{equation}
Write
\begin{equation}\label{eq:H-plane}
\pi^{\ast}H
\sim_{\Q}
b_HL-\sum_{i=0}^{m+3}\mu_i e_i.
\end{equation}
Equations \eqref{eq:Q-plane} and $\pi^{\ast}H\cdot\widetilde Q=0$ give
\begin{equation}\label{eq:plane-relations}
2b_H=\sum_{i=0}^{m+3}\mu_i.
\end{equation}
Moreover,
\begin{equation}\label{eq:plane-positive}
\mu_i=\pi^{\ast}H\cdot e_i>0,
\qquad
b_H-\mu_i-\mu_j
=\pi^{\ast}H\cdot(L-e_i-e_j)>0
\quad(i\neq j).
\end{equation}

Fix $r\in\{0,\dots,m+3\}$. For $j\neq r$,  let $\widetilde E_{rj}$ be the strict transform of the secant line through $x_r$ and $x_j$. Then $\widetilde E_{rj}\sim L-e_r-e_j$. The curves $\widetilde E_{rj}$, $j\neq r$, are pairwise disjoint $(-1)$-curves and satisfy $\widetilde Q\cdot\widetilde E_{rj}=0$. Let
\[
\varphi_r\colon\widetilde S_m^{m+3}\longrightarrow Z_r
\]
denote the contraction of these curves.

The pencil of lines through $x_r$ has reducible fibers $\widetilde E_{rj}+e_j$. After contracting the components $\widetilde E_{rj}$, the pencil induces a $\P^1$-bundle $\tau_r\colon Z_r\to \P^1$. Let $\overline{F}_r$ be the fiber class of $\tau_r$ and put $\widetilde F_r\coloneqq \varphi_r^{\ast}\overline{F}_r\sim L-e_r$ and $\overline{Q}_r\coloneqq (\varphi_r)_{\ast}(\widetilde Q)$. Then $\overline{Q}_r^2=-m$, $\overline{F}_r^2=0$, and $\overline{Q}_r\cdot\overline{F}_r=1$. Hence, $Z_r\simeq\F_m$. On $\widetilde S_m^{m+3}$, we have
\begin{equation}\label{eq:alternate-presentation}
e_j\sim \widetilde F_r-\widetilde E_{rj}
\quad(j\neq r),
\qquad
e_r\sim\widetilde Q+(m+1)\widetilde F_r-\sum_{j\neq r}\widetilde E_{rj}.
\end{equation}
Set $\alpha_r\coloneqq b_H-\mu_r$ and $c_{rj}\coloneqq b_H-\mu_r-\mu_j$ for $j\neq r$.
Using \eqref{eq:plane-relations} and \eqref{eq:alternate-presentation}, one obtains
\begin{equation}\label{eq:H-alternate}
\pi^{\ast}H
\sim_{\Q}
\alpha_r\widetilde Q+m\alpha_r\widetilde F_r-
\sum_{j\neq r}c_{rj}\widetilde E_{rj}.
\end{equation}
All $c_{rj}$ are positive by \eqref{eq:plane-positive}.

Choose distinct indices $u_r,v_r\neq r$ such that $\mu_{u_r}$ and $\mu_{v_r}$ are respectively the largest and second largest elements of the multiset $\{\mu_j\mid j\neq r\}$. Put
\[
\beta_r\coloneqq\mu_{v_r},
\qquad
q_r\coloneqq b_H-\mu_r-\beta_r,
\qquad
\gamma_r\coloneqq b_H-\mu_{u_r}-\mu_{v_r},
\]
and
\[
A_r\coloneqq\{j\neq r\mid\mu_j>\beta_r\},
\qquad
I_r^0\coloneqq\{j\neq r\mid\mu_j=\beta_r\},
\qquad
B_r\coloneqq\{j\neq r\mid\mu_j<\beta_r\}.
\]
Then $q_r,\gamma_r>0$. For every $j\neq r$,
\[
c_{rj}<q_r\Longleftrightarrow j\in A_r,
\qquad
c_{rj}=q_r\Longleftrightarrow j\in I_r^0,
\qquad
c_{rj}>q_r\Longleftrightarrow j\in B_r.
\]
Equation \eqref{eq:gamma-general} gives
\[
\gamma_{q_r}=m\alpha_r-
\sum_{j\in B_r}c_{rj}-
q_r\bigl(m+1-|B_r|\bigr)
=\gamma_r.
\]
Let $\overline{T}_r$ be the fiber of $\tau_r$ induced by the tangent line to $\rho(\widetilde Q)$ at $x_r$, let $\widetilde T_r\coloneqq \varphi_r^{\ast}\overline{T}_r$, and put $\Theta_r\coloneqq \pi_{\ast}(\widetilde T_r)$. The tangent line contains none of the points $x_j$ with $j\neq r$. Hence, $\overline{T}_r$ contains none of the blow-up centers $\varphi_r(\widetilde E_{rj})$. 

Apply Proposition \ref{prop:danielewski} to the birational morphism $\varphi_r\colon \widetilde S_m^{m+3}\to Z_r\simeq \F_m$, with section $e_r$, exceptional curves
$\widetilde E_{rj}$ for $j\neq r$, fiber class
$\widetilde F_r$, coefficient $m\alpha_r$ in place of $a$,
coefficients $c_{rj}$ in place of $c_j$, and $q=q_r$.
Put $C_i\coloneqq\pi_{\ast}(e_i)$, $L_{rj}\coloneqq\pi_{\ast}(\widetilde E_{rj})$. Then we have
\begin{equation}\label{eq:D-last}
D_r
\coloneqq
q_rC_r
+\sum_{j\in A_r}(\mu_j-\beta_r)L_{rj}
+\sum_{j\in B_r}(\beta_r-\mu_j)C_j
+\gamma_r\Theta_r
\sim_{\Q}H,
\qquad
U_r=S_m^{m+3}\setminus\Supp(D_r).
\end{equation}

The points $\varphi_r(\widetilde E_{rj})$ lie on pairwise distinct fibers of $\tau_r$. After identifying $\P^1\setminus \{\tau_r(\overline{T}_r)\}$ with $\A_x^1$, denote the coordinates of the points $\tau_r(\varphi_r(\widetilde E_{rj}))$ by $\lambda_{rj}$. Proposition \ref{prop:danielewski} gives
\[
U_r
\simeq
\Spec\frac{\Bbbk[x,u,v]}
{\left(uv-\prod_{j\in I_r^0}(x-\lambda_{rj})\right)}.
\]
Since $v_r\in I_r^0$, the set $I_r^0$ is nonempty. Hence, $U_r$ is a smooth flexible Danielewski surface.

Each $U_r$ is a nonempty smooth flexible affine open subset of the surface $S_m^{m+3}$. Therefore, Lemma \ref{lem:connected-transversal} shows that the collection $\{U_0,\dots,U_{m+3}\}$ is transversal.

We next identify the divisorial components that occur in the boundary of every open subset $U_r$.

\begin{lemma}\label{lem:intersection-supports}
Every divisorial component of
\[
S_m^{m+3}\setminus\bigcup_{r=0}^{m+3}U_r
\]
is one of the curves $C_i$. Moreover,
\[
C_i\subseteq\bigcap_{r=0}^{m+3}\Supp(D_r)
\quad\Longleftrightarrow\quad
\#\{j\mid\mu_j>\mu_i\}\geq3.
\]
Consequently, at least three of the curves $C_i$ are not contained in $\bigcap_{r=0}^{m+3}\Supp(D_r)$.
\end{lemma}

\begin{proof}
Equation \eqref{eq:D-last} gives
\[
\Supp(D_r)
=C_r\cup \Theta_r
\cup\bigcup_{j\in A_r}L_{rj}
\cup\bigcup_{j\in B_r}C_j.
\]
The curves $\Theta_r$ are pairwise distinct and are distinct from every curve $L_{ij}$. For $s\notin\{r,j\}$, the curve $L_{rj}$ is not a component of $D_s$. Such an index $s$ exists since $m+4\geq6$. Therefore, only the curves $C_i$ can be contained in all the supports $\Supp(D_r)$.

For $r\neq i$,
\[
C_i\subseteq\Supp(D_r)
\quad\Longleftrightarrow\quad
\mu_i<\beta_r.
\]
The inequalities $\mu_i<\beta_r$ for all $r\neq i$ hold
exactly when at least three indices $j$ satisfy
$\mu_j>\mu_i$.
\end{proof}

\begin{lemma}\label{lem:no-divisor-last}
There is no $\Q$-divisor $D\sim_{\Q}H$ with
\[
\Supp(D)
\subseteq
S_m^{m+3}\setminus\bigcup_{r=0}^{m+3}U_r.
\]
\end{lemma}

\begin{proof}
By Lemma \ref{lem:intersection-supports}, write $D=\sum_{i\in I}a_iC_i$ and set $s\coloneqq\sum_{i\in I}a_i$. Since we have $\pi^{\ast}C_i=e_i+\frac1m\widetilde Q$, \eqref{eq:H-plane} gives $b_H=\frac{2s}{m}$ and $\mu_j=\frac{s}{m}=\frac {b_H}{2}$ for $j\notin I$.
By Lemma \ref{lem:intersection-supports}, the complement of $I$ contains at least three indices. For distinct $j,k\notin I$, one obtains $\pi^{\ast}H\cdot(L-e_j-e_k)=b_H-\mu_j-\mu_k=0$, contrary to \eqref{eq:plane-positive}.
\end{proof}

\begin{proof}[Proof of Theorem \ref{thm:main}]
Assume first that $0\leq n\leq m+1$. The affine open subsets constructed above are smooth, flexible, and $H$-polar, and they cover $(S_m^n)_{\rm reg}$. Proposition \ref{prop:flex-cover} therefore implies that $\Affcone_H(S_m^n)$ is flexible. This proves assertion \textup{(1)}.

Assume next that $n=m+2$. The collection $\mathcal{U}_H$ consists of smooth flexible $H$-polar affine open subsets. The assumptions of Corollary \ref{cor:open-orbit} follow from Proposition \ref{prop:danielewski} and Lemma \ref{lem:no-divisor-middle}. Hence, $\Affcone_H(S_m^{m+2})$ has an open $\SAut$-orbit in $\bigl(\Affcone_H(S_m^{m+2})\bigr)_{\rm reg}$.
This proves assertion \textup{(2)}.

Assume finally that $n=m+3$. The collection $\{U_0,\dots,U_{m+3}\}$ consists of smooth flexible $H$-polar affine open subsets. The assumptions of Corollary \ref{cor:open-orbit} follow from the construction of the $U_r$ and Lemma \ref{lem:no-divisor-last}. Hence, $\Affcone_H(S_m^{m+3})$ has an open $\SAut$-orbit in $\bigl(\Affcone_H(S_m^{m+3})\bigr)_{\rm reg}$. This proves assertion \textup{(3)} and completes the proof.
\end{proof}


\bibliographystyle{habbvr}
\bibliography{biblio}

\end{document}